\documentclass{amsart}
\usepackage{amsmath,amsfonts,amssymb,amsthm,comment}
\usepackage[all]{xy}

\newtheorem{Theorem}{Theorem}[section]
\newtheorem{Proposition}[Theorem]{Proposition} 
\newtheorem{Lemma}[Theorem]{Lemma}

\newtheorem{Corollary}[Theorem]{Corollary}

\newtheorem{Definition}[Theorem]{Definition}
\theoremstyle{definition}

\newtheorem{Question}{Question}
\newtheorem{Comment}[Theorem]{Comment}

\newcommand{\low}{\text{low}}

\newcommand{\wt}{\operatorname{wt}}

\newcommand{\bc}{\mathbb{C}}

\newcommand{\g}{\mathfrak{g}}

\newcommand{\Aa}{\mathcal{A}}
\newcommand{\LL}{\mathcal{L}}

\newcommand{\End}{\mbox{End}}

\newcommand{\Flip}{\mbox{Flip}}

\newcommand{\Id}{\text{Id}}

\newcommand{\C}{\mathcal{C}}
\newcommand{\modq}{{(mod \, q^{-1}\LL)}}

\theoremstyle{definition}

\begin{document}

\title[The crystal commutor and $\bar{R}$]{The crystal commutor and Drinfeld's unitarized $R$-matrix}

\author{Joel Kamnitzer}
\email{jkamnitz@aimath.org}
\address{American Institute of Mathematics\\ Palo Alto, CA}

\author{Peter Tingley}
\email{pwtingle@math.berkeley.edu}
\address{UC Berkeley, Department of Mathematics\\ Berkeley, CA}

\begin{abstract}
Drinfeld defined a unitarized R-matrix for any quantum group $U_q(\g)$. This gives a commutor for the category of $U_q(\g) $ representations, making it into a coboundary category.
Henriques and Kamnitzer defined another commutor which also gives $U_q(\g)$ representations the structure of a coboundary category. We show that a particular case of Henriques and Kamnitzer's construction agrees with Drinfeld's commutor. We then describe the action of Drinfeld's commutor on a tensor product of two crystal bases, and explain the relation to the crystal commutor. 
\end{abstract}

\maketitle


\section{Introduction}

Let $A$ and $B$ be the crystals of two representations of a simple complex Lie algebra $\g$. Using the Sch\"utzenberger involution, Henriques and Kamnitzer \cite{cactus} defined an isomorphism $\sigma_{A,B}: A \otimes B \rightarrow B \otimes A$, which they call the crystal commutor. This gives $\g$-crystals the structure of a coboundary category.

By an analogous construction, Henriques and Kamnitzer also defined a commutor $\sigma^{hk}_{V, W} : V \otimes W \rightarrow W \otimes V$, where $V$ and $W$ are finite dimensional representations of $U_q(\g)$. This gives $U_q(\g)$ representations the structure of a coboundary category. There is some choice in lifting the Sch\"utzenberger involution to representations, so the commutor here is not unique. 

There is a more standard isomorphism from $V \otimes W$ to $W \otimes V$, called the braiding. This is defined by $v \otimes w \mapsto \Flip \circ  R ( v \otimes w )$, where $R$ is the universal $R$ matrix. In \cite{D}, Drinfeld introduced a ``unitarized" $R$ matrix $\bar{R}$, and showed that the map $V \otimes W \rightarrow W \otimes V$ given by $v \otimes w \mapsto \Flip \circ  \bar{R} (v \otimes w)$ is a coboundary structure on the category of $U_q(\g)$ representations.

The first purpose of this note is to relate these two ways of putting a coboundary structure on the category of $U_q(\g)$ representations, thus answering a question from \cite{cactus}. We then show that, for any two crystal bases, Drinfeld's commutor preserves the tensor product of the lattices and acts by the crystal commutor on the tensor product of the bases (up to some negative signs).  Thus the crystal commutor is essentially a combinatorial limit of Drinfeld's commutor for representations.  This explains why the crystal commutor is a coboundary structure, and not a braiding, as one might naively expect.

This paper is organized as follows. In sections 2--6 we review some background material concerning the quantum Weyl group, crystal bases, and Drinfeld's unitarized R-matrix. In Section 7 we construct the unitarized $R$ matrix as $\bar{R}= (Y^{-1} \otimes Y^{-1}) \Delta(Y)$, where $Y$ belongs to a completion of $U_q(\g)$. In Section 8 we realize $\bar{R}$ as $(\xi'^{-1} \otimes \xi'^{-1}) \circ \xi'$, where $\xi'$ is a slight modification of Sch\"utzenberger involution. In Section 9 we describe how $\bar{R}$ acts on crystal bases. We finish with two questions.

\subsection{Acknowledgements}
We thank Arun Ram suggesting this project and for ideas which helped us get started.  We also thank Henning Andersen, Mark Haiman, Valerio Toledano Laredo, Nicolai Reshetikhin, Noah Snyder and Andrew Tolland for helpful discussions.  We thank the referee for helpful comments.  The second author thanks the American Institute of Mathematics for their support and UC Berkeley for their hospitality.

\section{Notation} \label{notation}
We must now fix some notation. For the most part we follow conventions from \cite{CP}.
\begin{itemize}
\item $\g$ is a complex simple Lie algebra with Cartan algebra $ \mathfrak{h} $, and $A = (a_{ij})_{i,j \in I}$ is its Cartan matrix. 

\item $ \langle \cdot , \cdot \rangle $ denotes the paring between $ \mathfrak{h} $ and $ \mathfrak{h}^\star $ and $ ( \cdot , \cdot) $ denotes the usual symmetric bilinear form on either $ \mathfrak{h}$ or $ \mathfrak{h}^\star $.  Fix the usual bases $ \alpha_i $ for $ \mathfrak{h}^\star $ and $ H_i $ for $\mathfrak{h}$, and recall that $ \langle H_i, \alpha_j \rangle = a_{ij} $.  

\item $ d_i = (\alpha_i, \alpha_i)/2 $, so that $ (H_i, H_j) = d_j^{-1} a_{ij} $.  Let $ B $ denote the matrix $ (d_j^{-1} a_{ij}) $. 

\item $ q_i = q^{d_i} $.  

\item $ H_\rho $ is the element of $ \mathfrak{h} $ such that $ \langle \alpha_i, H_\rho \rangle = d_i = (\alpha_i, \rho) $ for all $ i $.

\item $W$ is the Weyl group for $\g$, which is generated by the simple reflections $s_i$, for $i \in I$.

\item $\theta$ is the diagram automorphism such that $w_0 (\alpha_i) = - \alpha_{\theta(i)},$ where $w_0$ is the longest element in the Weyl group $W$.

\item $U_q(\g)$ is the quantized universal enveloping algebra associated to $\g$, generated over $\mathbb{C}(q)$ by $E_i$, $F_i$ for all  $i \in I$, and $K_H$ for $H$ in the co-weight lattice of $\g$. As usual, let $K_i= K_{H_i}.$  For convenience, we recall the exact formula for the coproduct:
\begin{equation} \label{coproduct}
\begin{cases}
\Delta{E_i} & = E_i \otimes K_i + 1 \otimes E_i \\
\Delta{F_i} &= F_i \otimes 1 + K_i^{-1} \otimes F_i \\
\Delta{K_i} &= K_i \otimes K_i
\end{cases}
\end{equation}

\item $[n]= \frac{q^n - q^{-n}}{q-q^{-1}},$ and $X^{(n)} = \frac{X^n}{[n][n-1] \cdots [2]}.$

\item $V_\lambda$ is the irreducible representation of $U_q(\g)$ with highest weight $\lambda$. Let $v_\lambda $ denote a fixed highest weight vector in this representation.
\end{itemize}

\section{The completion $\widetilde{U_q(\g)}$} \label{comp_section}

We will be working in the completion  $\widetilde{U_q(\g)}$ of $U_q(\g)$ with respect to the weak topology generated by all matrix elements of finite dimensional representations. This section includes two equivalent explicit definitions of $\widetilde{U_q(\g)}$ (Definition \ref{hatdef} and Corollary \ref{hatdef2}), as well as some basic results about its structure. Most importantly, we show that $\widetilde{U_q(\g)}$ is isomorphic to the direct product of the endomorphism rings of all $V_\lambda$. Thus an element of $\widetilde{U_q(\g)}$ is equivalent to a choice of $x \in \End ({V_\lambda})$ for each $\lambda \in P_+$. 

\begin{Definition} \label{hatdef} Let $R$ be the ring consisting of series $\sum_{k=1}^\infty X_k$, where each $X_k \in U_q(\g)$ and, for any fixed $\lambda$, $X_k \cdot V_\lambda = 0$ for all but finitely many $k$. Notice that there is a well defined action of $R$ on any $V_\lambda$. Let $I$ be the two sided ideal in  $R$ consisting of elements which act as zero on all $V_\lambda$. Then  $\widetilde{U_q(\g)}$ is defined to be $R/I$.
\end{Definition}

\begin{Comment}
This is equivalent to the completion with respect to the topology mentioned above, since $U_q(\g)$ is semi simple, so the set of matrix elements of finite dimensional representations is point-separating for $U_q(\g)$. In particular the natural map of $U_q(\g)$ to $\widetilde{U_q(\g)}$ is an embedding. 
\end{Comment}

This completion has a simple description as follows:

\begin{Theorem} \label{comp}
$\widetilde{U_q(\g)}$ is isomorphic as an algebra to $\displaystyle \prod_{\lambda \in P_+} \End_{\bc(q)} (V_\lambda)$.
\end{Theorem}

Before proving Theorem \ref{comp} we will need two technical lemmas.

\begin{Lemma} \label{isp}
There is an element $p_\lambda \in U_q(\g)$ such that 
\begin{enumerate}
\item $p_\lambda (v_\lambda) = v_\lambda$

\item For any $\mu \neq \lambda$, $p_\lambda$ sends the $\mu$ weight space of $V_\lambda$ to 0.

\item $p_\lambda V_\mu = 0$ unless $\langle \mu-\lambda, \rho^\vee \rangle > 0$ or $\mu = \lambda$.
\end{enumerate}
\end{Lemma}

\begin{proof}
Fix a lowest weight vector $v_\lambda^{\text{low}} \in V_\lambda$.
$V_\lambda$ is a quotient of $U_q^-(\g) \cdot v_\lambda$, so we can choose some $F \in U_q^-(\g)$ such that $Fv_\lambda = v_\lambda^{\text{low}}$. Similarly, we can choose some $E \in U^+_q(\g)$ such that $E v_\lambda^{\text{low}} = v_\lambda$. Then $p':= EF$ clearly satisfies the first two conditions.

For each $i \in I$, let $R_i = E_i^{( \langle \lambda, \alpha_i^\vee \rangle )} F_i^{( \langle \lambda, \alpha_i^\vee \rangle )}$. Let
\begin{equation*}
p_\lambda= \left(\prod_{i \in I} R_i\right) p',
\end{equation*}
where the product is taken in any order. It is straightforward to see that this element satisfies the desired properties.
\end{proof}

\begin{Lemma} \label{isend}
Let $I_\lambda$ be the kernel of the action of $U_q(\g)$ on $V_\lambda$. Then
$U_q(\g) /I_\lambda $ is isomorphic to $\End_{\bc (q) } V_\lambda$.
\end{Lemma}

\begin{proof}
Let $d = \dim (V_\lambda)$.
Using the PBW basis in the $E$s, there is a $d$ dimensional subspace $\mathcal{F}$ of elements in $U_q^+(\g)$ that act non-trivially on $V_\lambda,$ and in fact such that  $p_\lambda \mathcal{F}$ is still $d$ dimensional, where  $p_\lambda$ is as in Lemma \ref{isp}. One can tensor this space with the PBW operators from $U_q^-(\g)$ to get a $d^2$ dimensional subspace of $U_q(\g)$ that acts non-trivially on $V_\lambda$. The result follows.
\end{proof}

\begin{proof}[{\bf Proof of Theorem \ref{comp}}]
Using Lemmas \ref{isp} and \ref{isend}, we can realize any endomorphism of $V_\lambda$ using an element of $U_q(\g)$ that kills $V_\mu$ unless $\langle \mu-\lambda, \rho^\vee \rangle >0$, or $\mu = \lambda$. The result follows.
\end{proof}

We include the following result to show how our definition of $\widetilde{U_q(\g)}$ relates to other completions that appear in the literature. This could also be taken as the definition of $\widetilde{U_q(\g)}$.

\begin{Corollary} \label{hatdef2}
Let each $\lambda \in P_+$, let $I_\lambda$ be the two sided ideal of $U_q(\g)$ generated by all $E_i^{\langle \lambda, \alpha_i^\vee \rangle}$ and $F_i^{\langle \lambda, \alpha_i^\vee \rangle}$. Let $\displaystyle U_q''(\g)= \lim_{\leftarrow} U_q(\g)/I_\lambda$, using the partial order on weights where $\mu \leq \lambda$ if and only if $\lambda - \mu \in P_+$. $U_q''(\g)$ acts in a well defined way on any finite dimensional module, so there is a map $U_q''(\g) \rightarrow \widetilde{U_q(\g)}$. This is an isomorphism.
\end{Corollary}

\begin{proof} \label{whdef2}
The same argument as we used to prove Theorem \ref{comp} shows that the image is $\displaystyle \prod_{\lambda \in P_+} \End(V_\lambda)$, which is all of $ \widetilde{U_q(\g)}$. The map is injective by the definition of $U''_q(\g)$.
\end{proof}

\begin{Comment}
 The completion $\widetilde{U_q(\g)}$ is related to the algebra $\dot{U}$ from \cite[Chapter 23]{L} as follows. $\dot{U}$ acts in a well defined way on each irreducible representation $V_\lambda$, and no non-zero element of $\dot{U}$ acts as zero on every $V_\lambda$. Hence $\dot{U}$ naturally embeds in $\widetilde{U_q(\g)}$. There is a canonical basis $\dot{B}$ for $\dot{U}$. All but finitely many elements of $\dot{B}$ act as zero on any given $V_\lambda$ (see \cite{L} Remark 25.2.4 and Section 23.1.2), so the space of all formal (infinite) linear combinations of elements of $\dot{B}$ also maps to $\widetilde{U_q(\g)}$. This map is bijective, and so $\widetilde{U_q(\g)}$ is naturally identified with the space of all formal linear combinations of elements of $\dot{B}$.
 \end{Comment}

It is clear the  $\widetilde{U_q(\g)}$ has the structure of a ring, and that it acts in a well defined way on finite representations. It also has a well defined topological coalgebra structure, with the coproduct of $u$ defined by the action of an element $u$ on a tensor product $V \otimes W$. This is only a topological coproduct because it maps $ \widetilde{U_q(\g)}$ into $\displaystyle \prod_{\lambda, \mu} \End_{\bc(q)} V_\lambda \otimes \End_{\bc(q)} V_\mu$, which can be though of as a completion of  $\displaystyle \prod_{\lambda} \End_{\bc(q)} V_\lambda \otimes  \prod_{\mu}  \End_{\bc(q)} V_\mu$. The restriction of this coproduct to $U_q(\g)$ agrees with the normal coproduct so, since $U_q(\g)$ is a dense subalgebra of $\widetilde{U_q(\g)}$, we see that $\widetilde{U_q(\g)}$ is a topological Hopf algebra. 

We will need to consider the group of invertible elements of $\widetilde{U_q(\g)}$ acting on $\widetilde{U_q(\g)}$ by conjugation. This action preserves the algebra structure of $\widetilde{U_q(\g)}$, but does not preserve the coproduct.

\begin{Definition} \label{Aactdef}
Let $X$ be an invertible element in $\widetilde{U_q(\g)}$. Define $C_X$ (conjugation by $ X $) to be the algebra automorphism of $\widetilde{U_q(\g)}$ defined by $u \rightarrow X u X^{-1}$.
\end{Definition}

\begin{Comment}
We caution the reader that $C_X$ is not that Hopf theoretic adjoint action of $X$, as defined in, for example, \cite{CP}. 
\end{Comment}

\begin{Comment} \label{thediagram}
For any invertible $X \in \widetilde{U_q(\g)},$ the action of $X $  on representations is compatible with the automorphism $ C_X $ in the sense that, for any representation $V$, the following diagram commutes:
\begin{equation*}
\xymatrix{
V  \ar@(dl,dr) \ar@/ /[rrr]^{X} &&& V \ar@(dl,dr) \\
\widetilde{U_q(\g)}  \ar@/ /[rrr]^{C_X} &&& \widetilde{U_q(\g)}. \\
}
\end{equation*} 
In general, $C_X$ does not preserve the subalgebra $U_q(\g)$ of $\widetilde{U_q(\g)}$, although it does in all cases we consider here.   
\end{Comment}

\subsection{Coalgebra antiautomorphisms}
We will be particularly interested in the case where $C_X $ is a coalgebra antiautomorphism. Explicitly, this means that $C_X$ satisfies the equation 
\begin{equation} \label{coanti}
\Delta^{\text{op}}(C_X(u)) = C_X \otimes C_X \big( \Delta(u) \big), \ \text{ for all } u \in U_q(\g).
\end{equation} 
Such $X$ are important because of the following result, which follows immediately from Comment \ref{thediagram} and the fact that $U_q(\g)$ is semi-simple.
\begin{Proposition} \label{comcon}
$C_X$ is a coalgebra antiautomorphism if and only if the map
\begin{equation}
\begin{aligned}
V\otimes W &\rightarrow W \otimes V \\
v \otimes w &\mapsto \text{Flip} \circ (X^{-1} \otimes X^{-1})  \Delta(X) v \otimes w
\end{aligned}
\end{equation}
is an isomorphism of $U_q(\g) $ representations for all $V$ and $W$. \qed
\end{Proposition}

\section{Coboundary categories and the unitarized $R$-matrix} \label{pre2}
We now briefly review the universal $R$ matrix, and the corresponding braiding on the category of $U_q(\g)$ representations. We then give Drinfeld's definition of a coboundary category, and review his unitarization procedure whereby the universal $R$ matrix is modified, resulting in a coboundary structure on the category of $U_q(\g)$ representations.

\subsection{The R-matrix}
\begin{Definition} 
A braided monoidal category is a monoidal category $\C$, along with a natural isomorphism $\sigma^{br}_{V,W}: V \otimes W \rightarrow W \otimes V$ for each pair $V,W \in \C$, such that for any $U,V, W \in \C$, 
\begin{equation*} (\sigma^{br}_{U,W} \otimes \Id) \circ (\Id \otimes \sigma^{br}_{V,W}) = \sigma^{br}_{U \otimes V, W} \end{equation*}
\begin{equation*} (\Id \otimes \sigma^{br}_{U,W}) \circ (\sigma^{br}_{U,V} \otimes  \Id) = \sigma^{br}_{U,V\otimes W}. \end{equation*}
The system $\sigma^{br} : = \{ \sigma^{br}_{V,W} \}$ is called a braiding on $\C$.
\end{Definition}

We will use the term braiding for such a $ \sigma^{br} $ and use the term commutativity constraint for any natural isomorphism $ V \otimes W \rightarrow W \otimes V $. 

Let $\widetilde{U_q(\g) \otimes U_q(\g)}$ be the completion of $U_q(\g) \otimes U_q(\g)$ in the weak topology defined by all matrix elements of representations $V_\lambda \otimes V_\mu$.

\begin{Definition} \label{rR}
A universal $R$-matrix is an element $R$ of $\widetilde{U_q(\g) \otimes U_q(\g)}$ such that $ \sigma^{br}_{V,W} := \mbox{Flip} \circ R$ gives a braided structure to the monoidal category of $ U_q(\g) $ representations.
\end{Definition}

\begin{Comment}
The universal $R$ matrix is not truly unique. However, it exists, and there is a well studied standard choice.  We will use a result of Kirillov-Reshetikhin and Levendorskii-Soibelman (see Theorem \ref{sR}) which describes this standard R-matrix in terms of $T_{w_0} $.
\end{Comment}

\subsection{Coboundary categories}
An analogous notion to braided monoidal categories is that of a coboundary monoidal category, due to Drinfeld  \cite[Section 3]{D}.

\begin{Definition} \label{cobdef}
A coboundary monoidal category is a monoidal category $\C$, along with a natural isomorphism $\sigma_{V,W}: V \otimes W \rightarrow W \otimes V$ for each pair $V,W \in \C$, satisfying
\begin{enumerate}

\item \label{commutor1} $\sigma_{W,V} \circ \sigma_{V,W} = \Id.$

\item \label{commutor2} For all $U,V,W \in \C$, the following diagram commutes:
\begin{equation*} \label{comp_diag}
\xymatrix{
U \otimes V \otimes W \ar@/ /[rr]^ {\hspace{-8pt} \sigma_{U,V} \otimes \Id} 
\ar@/ /[d]^{\Id \otimes \sigma_{V,W}} && V \otimes U \otimes W  \ar@/ /[d]^{\sigma_{V \otimes U, W}}\\
U \otimes W \otimes V 
\ar@/ /[rr]^ {\hspace{-8pt} \sigma_{U, W \otimes V}} && W \otimes V \otimes U.
 }
\end{equation*}

\end{enumerate}
\end{Definition}

Following \cite[Section 3]{cactus}, we call (\ref{commutor1}) the symmetry axiom and (\ref{commutor2}) the cactus axiom.  We will use the term commutor for a $ \sigma $ that satisfies these two conditions.

Though the braiding $ \sigma^{br} $ is better known, the category of $ U_q(\g) $ modules also has a natural commutor $ \sigma^{dr} $, which is our main object of study.  We now review its construction following Drinfeld \cite[Section 3]{D} and Berenstein-Zwicknagl \cite[Section 1]{BZ}.

\subsection{The unitarized R-matrix} \label{R}
Consider the ``ribbon" or ``quantum Casimir element'' $ Q \in \widetilde{U_q(\g)} $, which acts  on the irreducible representation $ V_\lambda$, as multiplication by $ q^{(\lambda, \lambda + 2 \rho)} $ (see for example \cite{BK}). In fact, $Q$ can act by fractional powers of $q$, so to be precise, we should adjoin a fixed $k^{th}$ root of $q$ to our base field $\bc(q)$, where $k$ is twice the dual Coxeter number for $\g$. 

\begin{Proposition}[see \cite{BK}, Section 2.2] \label{ror}
$R^{op} R = Q^{-1} \otimes Q^{-1} \Delta(Q).$
\end{Proposition}

The element $ Q $ is central, and admits a central square root, denoted  $ Q^{1/2} $, which  acts on $V_\lambda$ as multiplication by the constant $ q^{(\lambda, \lambda)/2 + (\lambda, \rho) }$.  $\bar{R}$  is defined as
\begin{equation*}
\bar{R} := R (R^{op}R)^{-1/2}.
\end{equation*}
Using Proposition \ref{ror} and the fact that $Q^{1/2}$ is central, this is equivalent to
\begin{equation} \label{eq:thetaR}
\bar{R} = R (Q^{1/2} \otimes Q^{1/2}) \Delta(Q^{-1/2}) = (Q^{1/2} \otimes Q^{1/2}) R \Delta(Q^{-1/2}),
\end{equation}

\begin{Definition}
Define the commutor for the category of $U_q(\g) $-modules by $ \sigma^{dr} := \text{Flip} \circ \bar{R} $.
\end{Definition}

The following is an easy consequence of the definitions.
\begin{Proposition}[{\cite[Proposition 3.3]{D}}]
$\sigma^{dr}$ is a coboundary structure on the category of $ U_q(\g) $ modules, ie. it satisfies the conditions of Definition \ref{cobdef}.
\end{Proposition}

\section{The quantum Weyl group} \label{qWeyl}
Following Lusztig \cite[Part VI]{L} and \cite[Section 5]{L2}, we introduce an action of the braid group of type $ \g $ on any $V_\lambda$, and thus a map from the braid group to $\widetilde{U_q(\g)}$. The images of elements of the braid group are invertible elements in $\widetilde{U_q(\g)}$, so, as discussed in Section \ref{comp_section}, we can define an action of the braid group on $\widetilde{U_q(\g)}$ by conjugation. This action in fact restricts to an action of the braid group on $U_q(\g)$.  

\subsection{The definition}
We first define the action of the generators $T_i$. Our conventions are such that $T_i$ is $T_{i,-1}^{''} = T_{i,1}'^{-1}$ in the notation from \cite{L}. 
 \begin{Definition}[{see \cite[5.2.1]{L}}] \label{defti} 
 $T_i$ is the element of  $\widetilde{U_q(\g)}$ that acts on a weight vector $v$  by:
 \begin{equation*}
 T_i(v)= \sum_{\begin{array}{c} a,b,c \geq 0 \\ a-b+c=(\wt(v), \alpha_i) \end{array}} (-1)^b q_i^{ac-b}E_i^{(a)} F_i^{(b)} E_i^{(c)} v.
 \end{equation*}
\end{Definition}

By \cite[Theorem 39.4.3]{L}, these $T_i$ generate an action of the braid group on each $V_\lambda$, and thus a map from the braid group to $\widetilde{U_q(\g)}$. This realization of the braid group is often referred to as the quantum Weyl group. It is related to the classical Weyl group by the fact that, for any weight vector $v \in V$, $\wt(T_i(v))= s_i(\wt(v))$.

\begin{Theorem} [see \cite{CP}, Theorem 8.1.2 or \cite{L}, Section 37.1.3]
The conjugation action of the braid group on $\widetilde{U_q(g)}$  (see Definition \ref{Aactdef}) preserves the subalgebra $U_q(\g)$, and is defined on generators by:
\begin{equation} \label{defT}
\begin{cases}
C_{ T_i} (E_i)= -F_iK_i \\
C_{T_i} (F_i)= -K_i^{-1} E_i \\
C_{T_i} (K_H)= K_{s_i(H)} \\
C_{T_i} (E_j) = \sum_{r=0}^{-a_{ij}} (-1)^{r-a_{ij}} K_i^{-r} E_i^{(-a_{ij}-r)} E_j E_i^{(r)} \mbox{  if  } i \neq j \\
C_{T_i} ( F_j) = \sum_{r=0}^{-a_{ij}} (-1)^{r-a_{ij}} K_i^{r} F_i^{(r)} F_j F_i^{(-a_{ij}-r)} \mbox{  if  } i \neq j.
\end{cases}
\end{equation}
\end{Theorem}

Fix some $w$ in the Weyl group $W$, and a reduced decomposition of $w$ into simple reflections $w= s_{i_1} \cdots s_{i_k}$. By \cite[Section 2.1.2]{L}, the element $T_w \in \widetilde{U_q(\g)} $ defined by
\begin{equation} \label{braidw}
T_w:= T_{i_1} \cdots T_{i_k}
\end{equation}
is independent of the reduced decomposition. Furthermore, the following holds.

\begin{Lemma} [see \cite{CP} Proposition 8.1.6]
Let $w \in W$ be such that $w(\alpha_i)= \alpha_j$. Then $C_{T_w} (E_i)= E_j$.
\end{Lemma}

\subsection{The action of $T_{w_0}$}

Let $w_0$ be the longest element of the Weyl group, and $T_{w_0}$ the corresponding element of the braid group given by Equation (\ref{braidw}).

\begin{Lemma} \label{stillE}
The action of $C_{T_{w_0}}$ on $U_q(\g)$ is given by
\begin{equation*}
\begin{cases}
C_{T_{w_0}} (E_i) = -F_{\theta(i)} K_{\theta(i)} \\
C_{T_{w_0}} (F_i) = -K_{\theta(i)}^{-1} E_{\theta(i)} \\
C_{T_{w_0}} (K_H) = K_{w_0(H)}, \mbox{ so that } C_{T_{w_0}} (K_i)  = K_{\theta(i)}^{-1}
\end{cases}
\end{equation*}
\end{Lemma}

\begin{proof}
Fix $i$. Then $T_{w_0}$ can be written as $T_{w_0} = T_{\theta(i)} T_w$ for some $w$ in the Weyl group. By the definition of $\theta$,  $C_{T_{w_0}} (E_i)$ is in the weight space $-\alpha_{\theta(i)}$. It follows that $C_{T_w}(E_i)$ is in the weight space $\alpha_{\theta(i)}$. Hence by Lemma \ref{stillE}, $C_{T_w}(E_i)= E_{\theta(i)}$. Therefore, by (\ref{defT}), $C_{T_{w_0}} (E_i)=  -F_{\theta(i)} K_{\theta(i)} $, as required. A similar proof works for $F_i$.  The action on $ K_H $ is straightforward.
\end{proof}

\begin{Comment}
Note that $ C_{T_{w_0}} $ is not a coalgebra antiautomorphism, so we cannot use $ T_{w_0} $ to construct a commutativity constraint in the manner of Proposition \ref{comcon}.  We will first need to correct $ T_{w_0} $.  There are essentially two natural ways of doing this --- one leads to the standard braiding and the other to Drinfeld's coboundary structure.
\end{Comment}

We now understand the action of $C_{T_{w_0}}$ on $U_q(\g)$. We also need to understand how $T_{w_0}$ acts on any finite dimensional representation and in particular on highest weight vectors.

\begin{Lemma} \label{th:Ti}
Let $ V $ be any representation, and $ v \in V $a weight vector such that $ E_i \cdot v = 0 $.  Then $ T_i(v) =  (-1)^n q^{d_i n} F_i^{(n)} v $, where $ n = \langle \wt(v), \alpha_i^\vee \rangle $. 
\end{Lemma}

\begin{proof}
Fix $v \in V$ with $E_i(v)=0$, and let $n= \langle \wt(v), \alpha_i^\vee \rangle $. It follows from $U_q(sl_2)$ representation theory that $F_i^{n+1} (v)=0$. The lemma then follows directly from the definition of $ T_i $ (Definition \ref{defti}).
\end{proof}

The following can be found in  \cite[Lemma 39.1.2]{L}  recalling that our $T_i$ is equal to $T_{i,1}'^{-1}$ in the notation from that book, although we find it convenient to include a proof. 

\begin{Proposition}  \label{jjthings}
Let $ w= s_{i_1} \cdots s_{i_\ell} $ be a reduced word. For each $1 \leq k \leq \ell$, the following statements hold.
\begin{enumerate}

\item \label{jj1} $E_{i_{k+1}} T_{i_k} \cdots T_{i_1} (v_\lambda) = 0$.

\item \label{jj2} $T_{i_k} \cdots T_{i_1} (v_\lambda) = (-1)^{n_1+ \cdots +n_k} q^{d_{i_1} n_1 + \dots + d_{i_k} n_k} F_{i_k}^{(n_k)} \cdots F_{i_1}^{(n_1)} v_\lambda$,\\ where $n_j = \langle  s_{i_1} \cdots s_{i_{j-1}}\alpha_{i_j}^\vee, \lambda \rangle $. 
\end{enumerate}
\end{Proposition}

\begin{proof}
Note that $\wt(E_{i_{k+1}} T_{i_k} \cdots T_{i_1} (v_\lambda)) = s_{i_{k}} \cdots s_{i_1} \lambda + \alpha_{i+1}$, so it suffices to show that the dimension of the $ s_{i_{k}} \cdots s_{i_1} \lambda + \alpha_{i_{k+1}}$ weight space in $V_\lambda$ is zero. The dimensions of weight spaces are invariant under the Weyl group, so we may act by $s_{i_1} \cdots s_{i_k}$ and instead show that the $\lambda +s_{i_1} \cdots s_{i_k} \alpha_{i_{k+1}} $ weight space of $V_\lambda$ is zero. But $s_{i_1} \cdots s_{i_k} s_{i_{k+1}}$ is a reduced word in the Weyl group, which implies that $s_{i_1} \cdots s_{i_k} \alpha_{i_{k+1}} $ is a positive root. Since $\lambda$ is the highest weight of $V_\lambda$, part  (\ref{jj1}) follows.
  
Part (\ref{jj2}) Follows by repeated use of (\ref{jj1}) and Lemma \ref{th:Ti}.
\end{proof}

\begin{Definition} \label{lowdef}
Fix a highest weight vector $v_\lambda \in V_\lambda$. Define the corresponding lowest weight vector $ v_{\lambda}^{\text{low}} \in V_\lambda $ by 
\begin{equation} \label{twl} T_{w_0} v_\lambda = (-1)^{\langle 2 \lambda, \rho^\vee \rangle} q^{(2\lambda, \rho)} v_{\lambda}^{\text{low}}. \end{equation}
\end{Definition}

\begin{Proposition} \label{toptobottom}
For any reduced expression $ w_0 = s_{i_1} \cdots s_{i_m} $, we have
\begin{equation*}
v_{\lambda}^{\text{low}} = F_{i_m}^{(n_m)} \cdots F_{i_1}^{(n_1)} v_\lambda,
\end{equation*}
where $ n_j = \langle  s_{i_1} \cdots s_{i_{j-1}}\alpha_{i_j}^\vee, \lambda \rangle. $
\end{Proposition}

\begin{proof}
Note that $ T_{w_0} v_\lambda = T_{i_m} \cdots T_{i_1} v_\lambda $.  The result then follows from 
Proposition \ref{jjthings} part (\ref{jj2}).
\end{proof}

\begin{Comment}
It follows from Proposition \ref{bottotop} below that $v_\lambda$ and $v_\lambda^{\low}$ are also related by
\begin{equation*}
T_{w_0} v_\lambda^{\low} =v_\lambda.
\end{equation*}
This is somewhat more difficult to prove directly.  It can also be shown that $v_\lambda^{\text{low}}$ is the lowest weight basis vector in the unique canonical (or global) basis for $V_\lambda$ containing $v_\lambda$.
\end{Comment}

\section{Crystal bases,  Sch\"utzenberger involution and the crystal commutor}
In this section, we introduce crystal bases, abstract crystals, the Sch\"utzenberger involution and the crystal commutor.  We also explore the relations between these topics. We follow \cite{K} for results on crystal bases and \cite{cactus} for results on the crystal commutor. Unfortunately, the conventions in \cite{CP} and \cite{K} do not quite agree, so we have modified some of the results from \cite{K} to match our conventions. In particular, we will need to work with crystal bases at $\infty$ instead of at $0$ since, with our choice of coproduct for $U_q(\g)$, crystal bases at $0$ do not have a nice tensor product.

\subsection{Crystal bases}

\begin{Definition}
Let $\Aa_\infty = \mathbb{C}[q]_\infty$ be the algebra of rational functions in one variable $q^{-1}$ over $\bc$ whose denominators are not divisible by $q^{-1}$.
\end{Definition}


\begin{Definition} \label{Kop} Fix a finite dimensional representation $V$ of $\g$, and $i \in I$.  Define the Kashiwara operators $ \tilde{F}_i, \tilde{E}_i : V \rightarrow V $ by linearly extending the following definition
\begin{equation*}
\begin{cases}
\tilde{F}_i(F_i^{(n)}(v)) = F_i^{(n+1)} (v) \\
\tilde{E}_i (F_i^{(n)}(v))= F_i^{(n-1)} (v).
\end{cases}
\end{equation*}
for all $ v \in V $ such that  $E_i (v)=0$. 
\end{Definition}

\begin{Comment} \label{symt}
It follows from the representation theory of $\text{sl}_2$ that  $ \tilde{F}_i$ and $ \tilde{E}_i $ can also be defined by linearly extending 
\begin{equation*}
\begin{cases}
\tilde{E}_i(E_i^{(n)}(v)) = E_i^{(n+1)} (v) \\
\tilde{F}_i (E_i^{(n)}(v))= E_i^{(n-1)} (v).
\end{cases}
\end{equation*}
for all $ v \in V $ such that  $F_i (v)=0$. Thus the operators are symmetric under interchanging the roles of $E_i$ and $F_i$, even if the definition does not appear to be. 
\end{Comment}

\begin{Definition}
A crystal basis of a representation $ V$ (at $q=\infty$) is a pair $(\LL, B) $, where $ \LL $ is an $\Aa_\infty$-lattice  of $ V$ and $B$ is a basis for $ \LL/q^{-1}\LL$, such that
\begin{enumerate}
\item $\LL $ and $ B $ are compatible with the weight decomposition of $ V $.
\item $\LL $ is invariant under the Kashiwara operators and $ B \cup 0 $ is invariant under their residues $ e_i := \tilde{E}_i^\modq, f_i := \tilde{F}_i^\modq : \LL/q^{-1}\LL \rightarrow \LL/q^{-1} \LL$.
\item For any $ b, b' \in B $, we have $ e_i b = b' $ if and only if $ f_i b' = b $.
\end{enumerate}
\end{Definition}

The following three theorems of Kashiwara are crucial to us.

\begin{Theorem}[\cite{K}, Theorem 1]
Let $ V, W $ be representations with crystal bases $ (\LL, A)$ and $(\mathcal{M}, B) $ respectively.  Then $( \LL \otimes \mathcal{M}, A \otimes B) $ is  a
crystal basis of $ V \otimes W $.
\end{Theorem}

\begin{Theorem}[\cite{K}, Theorem 2] \label{K2}  Let $ \LL_\lambda $ be the $ \Aa_\infty$ module generated by the $ \tilde{F}_i $ acting on $v_\lambda $ and let $ B_\lambda $ be the set of non-zero vectors in $ \LL_\lambda/q^{-1} \LL_\lambda $ obtained by acting on $v_\lambda$ with any sequence of $ \tilde{F}_i $. Then
$(\LL_\lambda, B_\lambda) $ is a crystal basis for $ V_\lambda $.
\end{Theorem}

Theorem \ref{K2} gives a choice of crystal basis for any $V_\lambda$, unique up to an overall scalar. The following result shows that these are all the crystal basis, and furthermore that any crystal basis of a reducible representation $V$ is a direct sum of such bases.

\begin{Theorem}[\cite{K}, Theorem 3] \label{uniquecb}
Let $ V $ be a representation of $U_q(\g) $ and let $ (\LL, B) $ be a crystal basis for $ V $.  Then there exists an isomorphism of $ U_q(\g) $ representations $ V \cong \oplus_{j} V_{\lambda_j} $ which takes $ (\LL, B) $ to $ (\oplus_j \LL_{\lambda_j}, \cup_j B_{\lambda_j}) $.
\end{Theorem}

\subsection{Abstract crystals}
It is often useful to work with the combinatorial data of $B$ along with the operators $\tilde{e}_i$ and $\tilde{f}_i$, without specifying how this arises as a crystal basis. This gives rise to the notion of (abstract) crystals.  

\begin{Definition}
An (abstract) crystal is a finite set $ B $ along with operators $ e_i, f_i : B \rightarrow B \cup \{0\} $ and a weight function $ \wt : B \rightarrow P $ which obey certain axioms (see \cite{cactus}).
\end{Definition}

Every crystal basis $ (\LL, B) $ gives an abstract crystal.  Namely, we choose $ B $ to be the underlying set and define $ e_i := \tilde{E}_i^\modq, f_i := \tilde{F}_i^\modq : B \rightarrow B $.  The weight map is defined using the decomposition of the crystal basis into weight spaces.

  There is well known tensor product rule for abstract crystals.  For abstract crystals $ A$ and $ B $, the underlying set of $A \otimes B$ is $ A \times B$ (whose elements we denote $ a \otimes b$) and the actions of $ e_i $ and $ f_i $ are given by the following rules:
\begin{equation} \label{edef} e_i (a \otimes b)=
\begin{cases}
e_i  (a) \otimes b, \quad \text{if}\quad  \varphi_i(a) \geq \varepsilon_i(b)\\
a \otimes e_i  (b),\quad \text{otherwise}
\end{cases} \end{equation}
\begin{equation} \label{fdef} f_i (a \otimes b)=
\begin{cases}
f_i  (a) \otimes b, \quad \text{if} \quad  \varphi_i(a) > \varepsilon_i(b)\\
a \otimes f_i (b),\quad \text{otherwise}.
\end{cases} \end{equation}
This is compatible with the notion of crystals arising from crystal bases, since, if $ (\LL, A) $ and $(\mathcal{M}, B) $ are crystal bases for two representations $ V, W $, then $ (\LL \otimes \mathcal{M}, A \otimes B) $ is a crystal basis for $ V \otimes W $ and the crystal corresponding to $ (\LL \otimes \mathcal{M}, A \otimes B ) $ is $ A \otimes B $ as defined above.

\subsection{The crystal commutor} \label{commutor}

From now on we only consider those crystals which come from crystal bases.  For those crystals, \cite[Section 2.2]{cactus} established the existence (and uniqueness) of a Sch\"utzenberger involution $ \xi_B : B \rightarrow B $, which satisfies the properties
\begin{equation*}
\xi_B(e_i \cdot b) = f_{\theta(i)} \cdot \xi_B(b), \quad \xi_B(f_i \cdot b) = e_{\theta(i)} \cdot \xi_B(b), \quad wt(\xi_B(b)) = w_0 \cdot wt(b).
\end{equation*}

Following a suggestion of A. Berenstein, the Sch\"utzenberger involution was used in \cite[Section 2.2]{cactus} to define the commutor for crystals by the formula
\begin{equation}
\begin{aligned} \label{commuter_definition}
  \sigma_{A,B} : A\otimes B &\rightarrow B \otimes A \\
 a \otimes b &\mapsto \xi(\xi(b) \otimes \xi(a)) = \mathrm{Flip} \circ \xi \otimes \xi (\xi (a \otimes b)).
 \end{aligned}
 \end{equation}
The second expression here is just the inverse of the first expression, and the equality is proved in \cite[Proposition 2]{cactus}.

\begin{Theorem}[\cite{cactus}, Theorem 6] $\g$-Crystals, with the above tensor product rule and commutor, forms a coboundary category.
\end{Theorem}

\subsection{Sch\"utzenberger involution on representations} \label{xionreps}

In the previous section we described Sch\"utzenberger involution $\xi$ as an involution on the crystal associated to a representation $V$ of $U_q(\g)$, and how $\xi$ is used to define the crystal commutor $\sigma$. We now describe (following \cite[2.4]{cactus}) how to modify this construction to obtain an involution of the actual representation $V$, and hence a commutor for $U_q(\g)$ representations. 

There is a one dimensional family of maps $ V_\lambda \rightarrow V_\lambda $ which exchange the actions of $ E_i, F_i $ with $ F_{\theta(i)}, E_{\theta(i)} $. Define $ \xi_{V_\lambda} $ to be the unique such map which takes the highest weight basis vector $v_\lambda $ to the lowest weight vector $v_\lambda^{\text{low}}$ (see Definition \ref{lowdef}). By Theorem \ref{comp}, these $\xi_{V_\lambda}$ combine to define an element $ \xi \in \widetilde{U_q(\g)} $. By construction $\xi$ is invertible and conjugation by $\xi $ is given by
\begin{equation*}
\begin{cases}
C_{\xi}(E_i)  = F_{\theta(i)}\\
C_{\xi}(F_i) = E_{\theta(i)}\\
C_{\xi}(K_H) = K_{w_0 \cdot H}.
\end{cases}
\end{equation*}
We can now define a commutor for the category of $U_q(\g)$ representations by, for any representations $V $ and $W$ of $U_q(\g)$,
\begin{equation*}
\sigma^{hk}_{V,W} := \xi_{W \otimes V} \circ (\xi_W \otimes \xi_V) \circ \mbox{Flip}  = \mathrm{Flip} \circ (\xi_V \otimes \xi_W) \circ \xi_{V \otimes W}.
\end{equation*}
Note that $ C_\xi $ is a coalgebra antiautomorphism, so, by Proposition \ref{comcon}, $ \sigma^{hk}_{V,W} $ is an isomorphism of $ U_q(\g) $ modules.  In fact, Henriques and Kamnitzer \cite[Theorem 4]{cactus} show the system of isomorphism $ \sigma^{hk}:= \{  \sigma^{hk}_{V,W}  \} $ is a coboundary structure (see Definition \ref{cobdef}). One main purpose of this paper is to examine the relationship between Drinfeld's commutor $ \sigma^{dr} $ and this $ \sigma^{hk} $ (or a slight modification thereof).

\subsection{Crystal bases and the Sch\"utzenberger involution} We now show that Sh\"utzenberger involution on representations, as defined in Section \ref{xionreps}, induces Sch\"utzenberger involution on crystal bases, as defined in Section \ref{commutor}.
We begin with the following lemmas.
\begin{Lemma} \label{lowlemma}
For any reduced word $ w_0 = s_{i_1} \cdots s_{i_m} $, we have that
\begin{equation} \label{adf1}
v_{\lambda}^{\text{low}} = \tilde{F}_{i_m}^{n_m} \cdots \tilde{F}_{i_1}^{n_1} v_\lambda \end{equation}
\begin{equation} \label{adf2}
v_{\lambda}= \tilde{E}_{\theta(i_m)}^{n_m} \cdots \tilde{E}_{\theta(i_1)}^{n_1} v_\lambda^\text{low} \end{equation}
where, as in Proposition \ref{jjthings}, $n_j = \langle  s_{i_1} \cdots s_{i_{j-1}}\alpha_{i_j}^\vee, \lambda \rangle $.
\end{Lemma}
\begin{proof}
By Proposition \ref{jjthings}, for each $0  \leq k <m$, we have $E_{i_{k+1}}F_{i_k}^{(n_k)} \cdots F_{i_1}^{(n_1)} v_\lambda=0$.  So Equation (\ref{adf1}) follows from the definition of the Kashiwara operators (Definition \ref{Kop}) and Proposition \ref{toptobottom}.

Now $w_0= s_{i_1} \cdots s_{i_m}$ is a reduced word in $W$, which implies that $s_{\theta(i_m)} \cdots s_{\theta(i_1)}$ is as well. Thus by Equation (\ref{adf1}),
\begin{equation} \label{anewe}
v_{\lambda}^{\text{low}} = \tilde{F}_{\theta(i_1)}^{\ell_m} \cdots \tilde{F}_{\theta(i_m)}^{\ell_1} v_\lambda,\end{equation}
where $\ell_j = \langle  s_{\theta(i_m)} \cdots s_{\theta(i_{m-j+2})}\alpha_{\theta(i_{m-j+1})}^\vee, \lambda \rangle $. For all $j$, 
\begin{equation*}
 s_{i_1} \cdots s_{i_{j-1}}\alpha_{i_j}^\vee  
 = -w_0 s_{\theta(i_1)} \cdots s_{\theta(i_{j-1})} \alpha_{\theta(i_j)}^\vee \\
=s_{\theta(i_m)} \cdots s_{\theta(i_{j+1})} \alpha_{\theta(i_j)}^\vee.
\end{equation*}
Here the first equality follows because for all $i \in I$, $w_0 (\alpha_i)= - \alpha_{\theta(i)}$ and $w_0 s_i w_0 = s_{\theta(i)}$. Thus $n_j = \ell_{m-j+1}$ so, by Equation (\ref{anewe}), 
\begin{equation*}
v_{\lambda}^{\text{low}} = \tilde{F}_{\theta(i_1)}^{n_1} \cdots \tilde{F}_{\theta(i_m)}^{n_m} v_\lambda.\end{equation*}
By Definition \ref{Kop}, this is equivalent to Equation (\ref{adf2}).
\end{proof}

The following result follows from \cite[Theorem 3.3 (b)]{L2}.  For completeness, we provide a proof.

\begin{Proposition}\label{bottotop} The action of the element $ \xi $ defined in Section \ref{xionreps} is an involution.  In particular
\begin{equation*} \xi(v_\lambda^{\low}) = v_\lambda. \end{equation*}
\end{Proposition}

\begin{proof}
Recall that $\xi_V$ interchanges the action of $E_i$ and $F_{\theta(i)}$. By Comment \ref{symt}, $\xi_V$ also interchanges the action of $\tilde{E}_i$ and $\tilde{F}_{\theta(i)}$. Thus, applying $\xi_{V_\lambda}$ to both sides of Equation (\ref{adf1}), 
\begin{equation*}
\xi(v_\lambda^\text{low})= \tilde{E}_{\theta(i_m)}^{n_m} \cdots \tilde{E}_{\theta(i_1)}^{n_1} v_\lambda^\text{low}.
\end{equation*}
The result then follows by Equation (\ref{adf2}).
\end{proof}

We can now describe how $ \xi_V $ acts on a crystal basis:  
\begin{Theorem} \label{xionbasis}
Let $ (\LL, B) $ be a crystal basis for a representation $ V $. Then the following holds.
\begin{enumerate}
\item  $ \xi_{V}(\LL) = \LL$. 
\item By (i), $ \xi_V $ gives rise to a map between $ \xi_V^\modq : \LL/q^{-1} \LL \rightarrow \LL/q^{-1}\LL $.  For each $ b \in B $, we have 
\begin{equation*} \xi^\modq_V(b) =  \xi_B(b) \end{equation*}
where $ \lambda $ is the highest weight of the crystal component containing $ b $.
\end{enumerate}
\end{Theorem}

\begin{proof}
First, we note that it is sufficient to prove the theorem in the case that $ (\LL,B) =(\LL_\lambda, B_\lambda) $.  The general case of the theorem then follows from an application of Theorem \ref{uniquecb}.

So  assume that $ V = V_\lambda, \LL = \LL_\lambda, B = B_\lambda $.  Note that $ \xi_{V_\lambda} $ exchanges the action of $ E_i $ and $ F_{\theta(i)} $ and hence, by Comment \ref{symt}, interchanges the actions of $\tilde{E}_i $ and $ \tilde{F}_{\theta(i)} $.  Since $ \LL $ is generated by $ \tilde{F}_i $ acting on $ v_\lambda$, we see that $ \xi_{V_\lambda}(\LL) $ is generated by $ \tilde{E}_i $ acting on $ \xi_{V_\lambda}(v_\lambda) = v_\lambda^{\text{low}} $.  Lemma \ref{lowlemma} shows that $ v_\lambda^{\text{low}} \in \LL $ so, since $ \LL $ is invariant under the action of the $ \tilde{E}_i $, we conclude that $ \xi_{V_\lambda}(\LL) \subset \LL $. By Proposition \ref{bottotop}, $\xi_{V_\lambda}^2$ is the identity, so in fact we must have $\xi_{V_\lambda}(\LL)=\LL$. 

For part (ii), note that $ v_\lambda^{\text{low}} $ is obtained by acting on $ v_\lambda $ with the $ \tilde{F}_i$.  Hence its reduction $ \text{mod} q^{-1}\LL $ must lie in $ B$, and in fact must by the lowest weight element in $B$.  The result follows because $C_\xi$ acts on  the set of Kashiwara operators according to $\tilde{F}_i \leftrightarrow \tilde{E}_{\theta(i)}$.
\end{proof}

\begin{Comment}
There is an even stronger connection between $ \xi_{V_\lambda} $ and the canonical (or global) basis $B_\lambda^c $ for $ V_\lambda $: It follows from \cite[Chapter 21]{L} that $ \xi_{V_\lambda} $ is the linear extension of the set map $ \xi_\lambda : B_\lambda^c \rightarrow B_\lambda^c $. However, because this fact does not hold for tensor products of canonical bases, it will not be useful for us. That is why we state the weaker fact above which holds for all crystal bases.
\end{Comment}

\section{Realizing $\bar{R}$ in the form $(Y^{-1} \otimes Y^{-1} ) \Delta{Y}$.} \label{RY}

We will construct the unitarized $R$ matrix in the desired form by modifying a similar result for the standard $R$ matrix, due to Kirillov-Reshetikhin and Levendorskii-Soibelman. Their result is stated as an expression in the $h$-adic completion of $U_h(\g) \otimes U_h(\g)$, although it in fact does give a well defined action on $V \otimes W$ for any representations $V$ and $W$ of $U_q(\g)$, so is well defined in $\widetilde{U_q(\g) \otimes U_q(\g)}$. In order to use Theorem \ref{sR}, for this section only, we will write some expressions in the $h$-adic completions of $U_h(\g)$ and  $U_h(\g) \otimes U_h(\g)$, and simply note that all the ones we use are well defined in $\widetilde{U_q(\g)}$ and $\widetilde{U_q(\g) \otimes U_q(\g)}$ as well. Translating the conventions in \cite{KR:1990} and \cite{LS} into ours, we obtain the following result.

\begin{Theorem}[{\cite[Theorem 3]{KR:1990}, \cite[Theorem 1]{LS}}] \label{sR}
With notation as in Section \ref{notation}, the standard $R$-matrix for $U_h(\g)$ can be realized as
\begin{equation*}
R=  \exp \Big( h \sum_{i, j \in I} (B^{-1})_{ij} H_i \otimes H_j \Big) (T_{w_0}^{-1} \otimes T_{w_0}^{-1}) \Delta(T_{w_0}).
\end{equation*}
\end{Theorem}

\begin{Definition} \label{defJ}
Let $J$ be the operator which acts on a finite dimensional representation $V$ of $U_q(\g)$ by multiplying each vector of weight $ \mu $ by $q^{( \mu, \mu) /2 + ( \mu, \rho )}$. It is a straightforward calculation to see that $J$ can be realized in a completion of  $U_h(\g)$ by
\begin{equation} \label{Jdefeq}
J:= \exp  \Big[ h \Big( \frac{1}{2} \sum_{i,j} \left( (B^{-1})_{ij} H_i H_j \right) + H_\rho \Big) \Big].
\end{equation}
\end{Definition}

Actually, $ (\mu, \mu)/2 + (\mu, \rho)$ can in some cases be a fraction. As in Section \ref{R}, we should really adjoin a fixed $k^{th}$ root of $q$ to our base field, with $k$ equal to twice the dual Coxeter number for $\g$. This causes no difficulty. 

\begin{Comment}
It follows from Lemma \ref{Dj} below that Theorem \ref{sR} is equivalent to saying $R= (X^{-1}\otimes X^{-1} ) \Delta(X)$, where $X= J T_{w_0}$.
\end{Comment}

\begin{Definition} \label{defY} $Y$ is the element in the completion of $ U_q(\g)$ defined by $Y:= Q^{-1/2} J T_{w_0} $.
\end{Definition} 
 
 We are now ready to state the main result of this section. 
\begin{Theorem} \label{answer1}
The unitarized $R$ matrix can be realized as
\begin{equation*}
\bar{R}= (Y^{-1} \otimes Y^{-1} ) \Delta(Y).
\end{equation*}
\end{Theorem}

\begin{Comment}
In fact, $Y$ is a well defined operator on $U_q(\g)$ over $\bc(q)$. That is, unlike for the standard $R$ matrix, we do not actually need to adjoin a $k^{th}$ root of $q$.
\end{Comment}

We prove Theorem \ref{answer1} by a direct calculation, using Theorem \ref{sR}. 
We will need the following technical lemma:

\begin{Lemma} \label{Dj}
\begin{equation*}
\Delta(J) = ( J \otimes J) \exp \Big( h\sum_{i,j \in I} (B^{-1})_{ij} H_i \otimes H_j \Big) 
\end{equation*}
\end{Lemma}

\begin{proof}
\begin{align*}
\Delta(J) &= \Delta \bigg(\exp  \bigg[ h \Big( \frac{1}{2} \sum_{i,j}  (B^{-1})_{ij} H_i H_j  + H_\rho \Big) \bigg] \bigg) \\
&= \exp \bigg[  h \bigg( \frac{1}{2}  \sum_{i,j} (B^{-1})_{ij} (H_i \otimes 1 + 1 \otimes H_i)(H_j \otimes 1 + 1 \otimes H_j)  + H_\rho \otimes 1 + 1 \otimes H_\rho \bigg) \bigg]\\
& = \nonumber \exp \Bigg[ h \bigg( \frac{1}{2}  \sum_{i,j} (B^{-1})_{ij} H_i H_j \otimes 1 + H_\rho \otimes 1 \bigg) \Bigg] \times \\
& 
\hspace{0.3in}  \times \exp \bigg[ h \bigg(  \frac{1}{2}  \sum_{i,j}  (B^{-1})_{ij} 1 \otimes H_i H_j  + 1 \otimes H_\rho \bigg) \bigg]  \exp \bigg[ h \sum_{i,j \in I} ( B^{-1})_{ij} H_i \otimes H_j  \bigg] \\
&=  (J \otimes J) \exp \Big[ h\sum_{i,j \in I} ( B^{-1})_{ij} H_i \otimes H_j \Big].
\end{align*}
\end{proof}



\begin{proof}[Proof of Theorem \ref{answer1}]
From the definition, we have that
\begin{align*}
(Y^{-1} \otimes Y^{-1} ) \Delta(Y) &= 
( T_{w_0}^{-1} J^{-1} Q^{1/2} \otimes  T_{w_0}^{-1} J^{-1} Q^{1/2} ) \Delta( Q^{-1/2}J T_{w_0}) \\
 &= ( T_{w_0}^{-1} \otimes T_{w_0}^{-1}) (J^{-1} \otimes J^{-1}) (Q^{1/2} \otimes Q^{1/2})  \Delta(Q^{-1/2})\Delta(J) \Delta(T_{w_0}) \\
 &= (Q^{1/2} \otimes Q^{1/2})  ( T_{w_0}^{-1} \otimes T_{w_0}^{-1}) (J^{-1} \otimes J^{-1}) \Delta(J) \Delta(T_{w_0}) \Delta(Q^{-1/2}),
 \end{align*}
where the last equality follows because $Q^{1/2}$ is central.
Then by Lemma \ref{Dj}:
\begin{align}
 \nonumber & (Y^{-1} \otimes Y^{-1} ) \Delta(Y)\\
 &= (Q^{1/2} \otimes Q^{1/2} )( T_{w_0}^{-1} \otimes T_{w_0}^{-1} ) 
\exp \Big( h \sum_{i, j \in I} (B^{-1})_{ij} H_i \otimes H_j \Big)
\Delta(T_{w_0}) \Delta(Q^{-1/2}) \\
\label{t13}&= (Q^{1/2} \otimes Q^{1/2}  ) 
\exp \Big( h \sum_{i, j \in I} (B^{-1})_{ij} H_i \otimes H_j \Big)
( T_{w_0}^{-1} \otimes T_{w_0}^{-1}) \Delta(T_{w_0}) \Delta(Q^{-1/2}) \\
\label{t14} &= (Q^{1/2} \otimes Q^{1/2}) R \Delta(Q^{-1/2}).
\end{align}
Here Equation (\ref{t13}) follows because $T_{w_0}$ permutes weight spaces as $w_0$, so $T_{w_0} H_i T_{w_0}^{-1} = -H_{\theta(i)}$. Equation (\ref{t14}) follows by Theorem \ref{sR}. The theorem follows by Equation (\ref{eq:thetaR}).
\end{proof}

\section{Realizing $\bar{R}$ using Sch\"utzenberger involution} \label{RS}

This section contains our first main result (Corollary \ref{answer2}), which realizes the unitarized $R$-matrix using a slight modification of Sch\"utzenberger involution. 
Fix two representations $V$ and $W$ of $U_q(\g)$. As discussed in Section \ref{xionreps}, there is a natural isomorphism $\sigma^{hk}_{V,W} : V \otimes W \rightarrow W \otimes V$ defined by
\begin{equation} \label{xx1}
\sigma^{hk}_{V,W} = \Flip \circ (\xi_V^{-1} \otimes \xi_W^{-1}) \circ \xi_{V \otimes W},
\end{equation}
where $\xi$ is Sch\"utzenberger involution. We have added inverses to the expression to make it more like Theorem \ref{answer1}. At the moment this has no effect, since $\xi$ is an involution, but it will be important later on. 

The commutor $\sigma^{hk}$ endows the Category of $U_q(\g)$ representations with a coboundary structure. In \cite{cactus}, Henriques and Kamnitzer note that one can multiply the action of $\xi$ on each irreducible representation by $\pm 1$, with the signs chosen independently for each $V_\lambda$, and Equation (\ref{xx1}) still defines a coboundary structure. They ask if there is a choice of signs such that the resulting commutor coincides with $\Flip \circ \bar{R}$, where $\bar{R}$ is the unitarized $R$ matrix. It turns out that we need  little bit more freedom. At the end of this section, we realize $\Flip \circ \bar{R}$ in terms of Sch\"utzenberger involution, where the action of $\xi$ on each irreducible representation is rescaled by certain $4^{th}$ roots of unity. It is convenient to first work with a different modification of $\xi$.

\begin{Definition} \label{xi11}
$ \xi'' $ is the element of $\widetilde{U_q(\g)}$ which acts on a weight vector $v \in V_\lambda$ by $ \xi''(v) = (-1)^{\langle\mu - w_0 ( \lambda), \rho^\vee\rangle } \xi(v)$, where $ \mu $ is the weight of $ v$. Notice, that $\mu - w_0 ( \lambda)$ is always in the root lattice, so $\langle\mu - w_0 ( \lambda), \rho^\vee\rangle$ is always an integer.
\end{Definition}

\begin{Proposition} \label{adanddiag}
$T_{w_0}, \xi''$ and $J$ are all invertible in $\widetilde{U_q(\g)}$, the actions of $C_{T_{w_0}}$, $C_{\xi''}$ and $C_{Q^{-1/2} J} = C_{J}$ all preserve the subalgebra $U_q(\g)$, and:
\begin{enumerate}

\item \label{adt} $
\begin{cases}
C_{T_{w_0}}(E_i) = -F_{\theta(i)} K_{\theta(i)} \\
C_{T_{w_0}}(F_i) = -K_{\theta(i)}^{-1} E_{\theta(i)} \\
C_{T_{w_0}}(K_H) = K_{w_0(H)}, \text{ so that } C_{T_{w_0}}(K_i) = K_{\theta(i)}^{-1}
\end{cases}
$

\item \label{adj} $
\begin{cases}
C_{J}(E_i) = K_i E_i  \\
C_{J}(F_i) = F_i K_i^{-1}\\
C_{J}(K_H) =  K_H
\end{cases}$

\item \label{adx}
$
\begin{cases}
C_{\xi''}(E_i)  = -F_{\theta(i)}\\
C_{\xi''}(F_i) = -E_{\theta(i)}\\
C_{\xi''}(K_H) = K_{w_0 \cdot H}
\end{cases}
$
\end{enumerate}
Furthermore, $ Y = \xi'' $ where, as in Section \ref{RY}, $Y= Q^{-1/2} J T_{w_0}$.
\end{Proposition}

\begin{proof}
It is clear from the definitions that these elements act as invertible endomorphisms on each $V_\lambda$, and hence by Theorem \ref{comp} they are invertible. Their conjugation actions preserve $U_q(\g)$ by (i), (ii) and (iii), which we prove below.

(\ref{adt}) This is Lemma \ref{stillE}.

(\ref{adj}) Let $v$ be a vector of weight $\mu$ in some finite dimensional representation. It is a straightforward calculation to see that $J(E_i v) = K_i E_i (J(v))$, $J(F_i(v))= F_i K_i^{-1} (J(v))$ and $J(K_H (v)) = K_{w_0 \cdot H} J(v)$.

(\ref{adx}) For $C_{\xi''} (E_i)$ and $C_{\xi''}( F_i)$ this follows immediately from Definition \ref{xi11} and the definition of Sch\"utzenberger involution (see Section \ref{xionreps}). It is a straightforward calculation to show that for any weight vector $v$, $\xi''(K_H (v)) = K_{w_0 (H)} \xi'' (v)$. It follows that $C_{\xi''} (K_H) = K_{w_0(H)}$.

It remains to show that $Y=\xi''$.
A direct calculation using (i), (ii) and (iii) shows that $C_{J} C_{T_{w_0}}= C_{\xi''}$. Since $Q^{-1/2}$ is central, this implies that $C_Y= C_{\xi''}$. 

For each $ \lambda$, there is a 1-dimensional family of endomorphisms of $ V_\lambda $ which are compatible with the automorphism $ C_Y = C_{\xi''} $ of $U_q(\g) $.  By Comment \ref{thediagram}, both $ Y $ and $ \xi'' $ give such endomorphisms.  Hence it suffices to check that they take the same value on one element of $ V_\lambda $, say the highest weight vector $ v_\lambda $.  

However, from the definition of $ Y$ and the definition of $v_\lambda^{low} $ (Definition \ref{lowdef}), we see that $Y(v_\lambda) = (-1)^{\langle 2\lambda, \rho^\vee \rangle} v_\lambda^\text{low}$.   On the other hand, from the definition of $ \xi'' $ we see immediately that $\xi''(v_\lambda) = (-1)^{\langle 2\lambda, \rho^\vee \rangle} v_\lambda^\text{low}$.
\end{proof}

The following corollaries give us the desired realization of the unitarized $R$ matrix.

\begin{Corollary} \label{Rxi11}
The unitarized R matrix acts on a tensor product $V \otimes W$ by
\begin{equation*}
\bar{R} (v \otimes w )= (\xi_V''^{-1} \otimes \xi_W''^{-1}) \circ \xi_{V \otimes W}'' (v \otimes w)
\end{equation*}
\end{Corollary}

\begin{proof}
This follows from Theorem \ref{answer1} since, by Proposition \ref{adanddiag}, $Y=\xi''$.
 \end{proof}

This is not quite what we were looking for since, for $v \in V_\lambda$, the relationship between $\xi(v)$ and $ \xi''(v) $ depends on the weight of $ v $, not just on $\lambda$.  To fix this problem, define $\xi' : V_\lambda \rightarrow V_\lambda $ by $ \xi'(v) = i^{2\langle\lambda, \rho^\vee\rangle} \xi(v)$. Notice that $\langle\lambda, \rho^\vee\rangle$ is in general only a half integer, so multiples of $i$ do appear.
We immediately deduce the following.
\begin{Corollary} \label{answer2}
The unitarized R matrix acts on a tensor product $V \otimes W$ by
\begin{equation*}
\bar{R} (v \otimes w) = (\xi_V'^{-1} \otimes \xi_W'^{-1}) \circ \xi_{V \otimes W}' (v \otimes w)
\end{equation*}
\end{Corollary}

\begin{proof}
Follows from Corollary \ref{Rxi11} by a straightforward calculation.
\end{proof}

\begin{Comment}
Notice that $\xi'_{V_\lambda} \circ \xi'_{V_\lambda} = (-1)^{\langle 2 \lambda, \rho^\vee \rangle} \Id$, and in particular $\xi'$ does not in general square to the identity, as required by Henriques and Kamnitzer. However, their argument can be modified slightly to show directly that $\Flip \circ  (\xi_V'^{-1} \otimes \xi_W'^{-1}) \circ\xi'_{V \otimes W}$ still defines a commutor which satisfies the axioms of a coboundary category. We do not include this, since it follows from the corresponding fact for  $\Flip \circ \bar{R}$.
\end{Comment}

\section{The action of $\bar{R}$ on a tensor product of crystal bases} \label{ocb} \label{RC}
This section contains our second main result, namely an explicit relationship between Drinfeld's commutor and the crystal commutor. Roughly, Theorem \ref{main2} shows that the crystal commutor is the crystal limit of Drinfeld's commutor (modulo some signs).  

Recall that if $ (\LL, B) $ is  a crystal basis for a representation $ V$, we have both a linear map $ \xi'_V : V \rightarrow V $ and a map of sets $ \xi_B : B \rightarrow B $ (coming from regarding $ B $ as an abstract crystal).  
The following proposition follows immediately from Theorem \ref{xionbasis}.

\begin{Proposition} \label{xi'onbasis}
Let $ (\LL, B) $ be a crystal basis for a representation $ V $.  Then:
\begin{enumerate}
\item  $ \xi'_V(\LL) = \LL$. 
\item By (i), $ \xi'_V $ gives rise to a map between $ {\xi'}_V^\modq : \LL/q^{-1} \LL \rightarrow \LL/q^{-1}\LL $.  For each $ b \in B $, we have 
\begin{equation*} {\xi'}^\modq_V(b) = i^{\langle \lambda, 2 \rho^\vee \rangle} \xi_B(b) \end{equation*} 
where $ \lambda $ is the highest weight of the crystal component containing $ b $. \qed
\end{enumerate}
\end{Proposition}

Now let $(\LL, A)$ and $(\mathcal{M}, B)$ be crystal bases for two finite dimensional representations $V$ and $W$ of $U_q(\g)$. The crystal commutor defines a map $\sigma_{A,B}: A \otimes B \rightarrow B \otimes A$.  This map comes from Drinfeld's commutor $ \sigma^{dr} = \Flip \circ \bar{R} $ in the following sense.

\begin{Theorem} \label{main2}
With the above setup:
\begin{enumerate}
\item $ \sigma^{dr}_{V,W}(\LL \otimes \mathcal{M}) = \mathcal{M} \otimes \LL $
\item By (i), $ \sigma^{dr}_{V,W}$ gives rise to a map 
\begin{equation*} {\sigma^{dr}_{V,W}}^{(mod \, q^{-1}(\LL \otimes \mathcal{M}))} : (\LL \otimes \mathcal{M}) / q^{-1} ( \LL \otimes \mathcal{M}) \rightarrow (\mathcal{M} \otimes \LL) / q^{-1} (\mathcal{M} \otimes \LL). 
\end{equation*} 
For all $ a \in A, b \in B$, 
\begin{equation*}
 {\sigma^{dr}_{V,W}}^{(mod \, q^{-1} (\LL \otimes \mathcal{M}))} (a \otimes b) = (-1)^{\langle \lambda + \mu - \nu, \rho^\vee \rangle}  \sigma_{A, B}(a \otimes b)
\end{equation*}
where $ \lambda, \mu$ and  $\nu $ are the highest weights of the components of $A, B$ and $A \otimes B$ containing $ a, b$ and $ a \otimes b $ respectively.
\end{enumerate}

\end{Theorem}

\begin{proof}
By Corollary \ref{answer2} and Proposition \ref{xi'onbasis}.i applied to the crystal bases $ (\LL, A), (\mathcal{M}, B)$ and $(\LL \otimes \mathcal{M}, A \otimes B) $:
\begin{equation*}
 \Flip \circ \bar{R}(\LL \otimes \mathcal{M}) = \text{Flip} \circ ({\xi'}_V^{-1} \otimes {\xi'}_W^{-1}) \circ \xi'_{V \otimes W} (\LL \otimes \mathcal{M}) = \text{Flip}(\LL \otimes \mathcal{M}) = \mathcal{M} \otimes \LL.
\end{equation*}
This establishes (i).

Similarly, (ii) follows directly from Corollary \ref{answer2} and Proposition \ref{xi'onbasis}.ii.
\end{proof}
 
Note that consistently working modulo the lattices, one can see that the coboundary properties of Drinfeld's commutor $ \sigma $ are transferred to the crystal commutor.  Of course it is very easy to prove the coboundary properties of the crystal commutor directly, but we feel this gives some explanation as to why these properties arise.

\section{Questions} \label{Questions}

We finish with a short discussion of some questions we feel merit further exploration.

\begin{Question}
For each connected subgraph $\Gamma$ of the Dynkin diagram, let $\g_\Gamma$ be the corresponding Levi subalgebra of $\g$, and $U_h(\g_\Gamma)$ be the corresponding Levi subalgebra of $U_h(\g)$. One can define $\xi''_\Gamma= Q_\Gamma^{-1/2} J_\Gamma T_{w_\Gamma}$, where $Q_\Gamma$ and $J_\Gamma$ act on a representation $V$ of $U_h(\g)$ via the obvious functor to $U_h(\g_\Gamma)$ representations, and $T_{w_\Gamma}$ is the braid group element corresponding to the longest word in the Weyl group of $\g_\Gamma$. These $\xi''_\Gamma$ are invertible, so they generate a group $\mathcal{W}$ acting on $U_h(\g)$ (which preserves only the algebra structure), and on representations of $U_h(\g)$. What group is this?

\end{Question}

We believe that answering this question would be an important step in understanding the relationship between the braid group and the cactus group. In the case $\g = \mbox{sl}_n$ we hope that $\mathcal{W}$ is closely related to the cactus group, where, if $\Gamma$ consists of nodes $s$ through $t-1$ of the Dynkin diagram, $\xi''_\Gamma$ corresponds to the generator $r_{[s,t]}$ of the cactus group (see \cite[Section 3]{cactus}). It cannot agree exactly since ${\xi''_\Gamma}^2 \neq \Id$ (only ${\xi''_\Gamma}^4 = \Id$).

\begin{Question}
The commutor gives an action of the $n$-fruit cactus group $J_n$, and hence its group algebra, on tensor products $V \otimes \cdots \otimes V$ of $U_q(\g)$ modules.  Does this action factor through any quotient algebra in some special cases?  For example what about the case when $ \g = \mathrm{sl}_n $ and each $ V $ is the standard representation.  What about the corresponding action on tensor products of crystals?
\end{Question}

For the case of the braiding, the corresponding question has a nice answer in the above special case.  The action of the braid group factors through the Hecke algebra and we have the quantum Schur-Weyl duality.

\end{document}